\documentclass[smallextended,referee,envcountsect]{svjour3}
\smartqed
\usepackage{graphicx}
\usepackage{lineno,hyperref,amsmath,longtable,booktabs,multirow,amssymb}
\usepackage{algorithmicx}
\usepackage{algpseudocode}

\usepackage{supertabular}
\usepackage[figuresright]{rotating}

\usepackage{algorithm}
\floatname{algorithm}{Algorithm}

\journalname{JOTA}

\begin{document}

\title{Outer-space  branch-and-bound algorithm for generalized linear
multiplicative programs}


\author{Bo Zhang}

\institute{Bo Zhang \at
                School of Civil and Hydraulic  Engineering, Ningxia University, Ningxia, Yinchuan 750021, P.R. China\\
                Tel.: +86-187-0960-2837\\
                \email{zbsdx121@163.com}
}

\date{Received: date / Accepted: date}

\maketitle

\begin{abstract}
This paper introduces a new global optimization algorithm for solving the generalized linear multiplicative problem (GLMP). The algorithm starts  by introducing $\bar{p}$ new variables and applying a logarithmic transformation to convert the problem into an equivalent problem (EP). By using the strong duality of linear program, a new convex relaxation subproblem is formulated to obtain the lower bounds for the optimal value of EP. This relaxation subproblem, combined with a simplicial branching process, forms the foundation of a simplicial branch-and-bound algorithm that can globally solve the problem.  The paper also includes an analysis of the theoretical convergence and computational complexity of the algorithm. Additionally, numerical experiments are conducted to demonstrate the effectiveness of the proposed algorithm  in various test  instances.
\end{abstract}
\keywords{Global optimization \and Linear  multiplicative program \and Branch-and-bound \and Outer approximation}
\subclass{90C26 \and  90C30}


\section{Introduction}

Consider the following  generalized linear  multiplicative program (GLMP):
 $$
({\rm GLMP})\left\{
\begin{aligned}
& \min ~ h(x)=
\prod_{j=1}^{p}(c_{j}^{\top}x+d_{j})^{\alpha_{j}}\\
&~{\rm s.t.,}~~x\in \mathcal{X}:=\{x\in \mathbb{R}^n|Ax\leq b\},
\end{aligned}
\right.$$
where the feasible domain $\mathcal{X}$ is $n$-dimensional, nonempty, and bounded; $p\geq2$, $A\in \mathbb{R}^{m\times n}$, $b\in \mathbb{R}^m$, $c_j\in \mathbb{R}^n $, $d_j\in \mathbb{R}$, $\alpha_{j}\in \mathbb{R}\backslash \{0\}$  and $ \min\limits_{x\in \mathcal{X}} c_{j}^{\top}x+d_{j}>0$.

As a class of non-convex programming problems, the GLMP problem covers many important applications, such as robust optimization \cite{4}, decision tree optimization \cite{5}, financial optimization \cite{2,1}, microeconomics \cite{3}, multiple-objective decision \cite{6,7},  optimal packing and layout \cite{9}  and VLSI chip design \cite{8}. Moreover, the bilinear program  and quadratic program  with rank 1 are special cases of the GLMP problem.  If all exponents are negative, the problem can be easily demonstrated as a convex optimization problem. Conversely, if all exponents are 1, the problem becomes NP-hard \cite{17}. From the non-convex perspective, the problem remains non-convex even if it has at least one positive exponent. As a result, it is not feasible to solve GLMP directly using convex optimization techniques.
Hence,
the study of the theoretical and computational aspects involved in discovering the global optimal solution for this problem presents substantial difficulties.

For addressing the GLMP problem, various ingenious methods have been developed, such as the heuristic methods \cite{32,31}, branch-and-bound (BB) algorithms \cite{23,5B,24,42,25,26}, outer-approximation methods \cite{27,28}, cutting-plane algorithm  \cite{33}, parameter simplex algorithm \cite{35},  fully polynomial time approximation algorithm  \cite{40}. Additionally, Shao and Ehrgott \cite{39} introduced a global optimization algorithm that utilizes the multi-objective linear program and its primal-dual conditions, specifically for GLMP problems with $\alpha_j=1$. Zhang et al. \cite{34}  proposed a BB algorithm that is based on linear relaxation, taking advantage of the multiplication structure of the objective function. Jiao et al. \cite{22}  utilized linear approximation methods for exponential and logarithmic functions to establish a BB algorithm for solving generalized linear multiplicative programs. Shen et al. \cite{24}  introduced an accelerated algorithm that combines the deletion technique with the BB scheme, enabling it to address a specific class of generalized linear multiplicative programs. In the case of GLMP problems with positive exponents, Liu and Zhao  \cite{38} extended the research results of Youness  \cite{37} and designed a level set algorithm. Benson \cite{41} developed a BB algorithm based on decomposition and Lagrangian duality theory to tackle linear programming problems with multiplicative constraints. For more complex mixed integer linear minimum multiplicative programs, Mahmoodian et al.  \cite{4n}  presented an algorithm based on multi-objective optimization theory. Recently, Jiao et al. \cite{5B,25B}  proposed two BB algorithms for GLMP, both utilizing linear relaxations and leveraging the properties of natural logarithmic functions.
 To gain a thorough understanding of problem GLMP,  readers can refer to \cite{38,4n,34}.

In this paper, we propose an outer-space simplicial BB algorithm (OSSBBA) to address the GLMP problem globally. Our algorithm begins by converting the GLMP problem into a different but equivalent problem (EP). This transformation allows us to readily create an underestimation of the objective function of EP and determine the lower bounds of its optimal value. Consequently, all the subproblems required to obtain these lower bounds are convex problems. As a result,  the main computational task of the algorithm revolves around solving a sequence of convex programs, with no added complexity in each iteration. In contrast to other BB algorithms such as \cite{5B,25B,24,42,25,26,45,34}, the proposed OSSBBA possesses the following unique characteristics.

a): Instead of using intricate polytopes \cite{24,42,25,26,34,5B,25B} as partition elements in the branch-and-bound search, OSSBBA employs simplices. Moreover, in OSSBBA, the branching occurs in $\mathbb{R}^{\bar{p}}$ ($\bar{p}\leq p\leq n$) rather than $\mathbb{R}^{p}$ \cite{5B,25B}  or $\mathbb{R}^{n}$ \cite{25,45}.
 This means that the algorithm may be able to save the required computational cost under certain circumstances.

b): Our analysis differs from the algorithms discussed in \cite{25,45,38} as it specifically examines the computational complexity of OSSBBA.  In the worst-case scenario, the estimated maximum number of iterations is lower than that of the method mentioned in \cite{34}. Additionally, the proof method employed in this literature differs significantly.

c): In view of the bounding technique of the BB search, we propose the utilization of a convex relaxation approach that exploits the distinctive characteristics of the problem. Conversely,  the approaches employed in Refs. \cite{34,5B,25B} employ diverse linear relaxations  to determine the lower bounds of the optimal value of the problem.

d): By examining Tables \ref{TB1}, \ref{TB2}, \ref{TB3}, \ref{TB4}, \ref{TB5}, \ref{TB6}, \ref{TB7}  and \ref{TB8}in Sect. \ref{SEZ5}, it becomes evident that OSSBBA significantly outperforms the algorithms in \cite{5B,25B} and the commercial software package BARON  \cite{Baron} in terms of CPU time for certain large-scale GLMP problems. Notably, the numerical findings demonstrate that OSSBBA efficiently discovers the optimal solution for all test  instances within a finite number of iterations and computational time.

The rest of this paper is structured as follows. In Sect. \ref{SEZ2}, we convert the GLMP problem into the EP problem and prove the related equivalence. Sect. \ref{SEZ3} introduces a convex programming problem that establishes a lower bound for the optimal value of EP. In Sect. \ref{SEZ4}, we propose a simplicial BB algorithm that combines convex relaxation and simplicial branching rule. Besides, we provide convergence and complexity analysis for this algorithm. Numerical comparison results are presented in Sect. \ref{SEZ5}.
This article concludes and provides a future perspective in the final section.

\section{Problem reformulation}\label{SEZ2}
 For addressing the GLMP problem globally, we equivalently lift the problem to a new nonlinear optimization model. To this end, by adopting the logarithmic transformation, GLMP  can
be converted into the following problem:
  $$
({\rm P})\left\{
\begin{aligned}
& \min ~ \nu(x):=\sum_{j=1}^{p}\alpha_j\ln (c_{j}^{\top}x+d_{j}) \\
&~{\rm s.t.}~~x\in \mathcal{X}.
\end{aligned}
\right.$$
Clearly, $\min\limits_{x\in \mathcal{X}}~ h(x)=\exp\left(\min\limits_{x\in \mathcal{X}}~ \nu(x)\right)$. This means that we can solve the problem P instead of directly solving the original problem GLMP.

Without loss of generality, throughout the whole article, let $$\begin{aligned}&J:=\{1,2,\cdots,p\}, ~J^+:=\{j\in J:\alpha_j>0\}=\{1,2,\cdots,\bar{p}\},\\
&~J^-:=\{j\in J:\alpha_j<0\}=\{\bar{p}+1,\bar{p}+2,\cdots,p\},\end{aligned}$$
where $\bar{p}\in \mathbb{N}$ and $1\leq \bar{p}\leq p$. Now, for all $j\in J^+$, let us compute
 \begin{equation}\label{Jeq1} 0<\underline{t}_j:=\frac{1}{\max_{x\in \mathcal{X}}c_{j}^{\top}x+d_{j}},~\overline{t}_j:
 =\frac{1}{\min_{x\in \mathcal{X}}c_{j}^{\top}x+d_{j}},
\end{equation}
 and  construct  the $\bar{p}$-simplex  \begin{equation}\label{Jeq2}
\mathcal{S}^0:=[v^{0,1},v^{0,2},\cdots,v^{0,\bar{p}+1}],\end{equation}
where $v^{0,1}=(\underline{t}_1,\cdots,\underline{t}_{\bar{p}})^{\top}$, $v^{0,j+1}=(\underline{t}_1,\cdots,\underline{t}_j+\bar{p}(\overline{t}_j-\underline{t}_j) \cdots,\underline{t}_{\bar{p}})^{\top}$, $j\in J^+$.

Using the rectangle $\mathcal{S}^0$,   P can be  transformed into the following form: $$
({\rm P_1})\left\{
\begin{aligned}
& \min ~ \phi(x,t)\\
&~{\rm s.t.}~~x\in \mathcal{X}, t\in \mathcal{S}^0,
\end{aligned}
\right.$$
where
$t=(t_1,t_2,\cdots,t_{\bar{p}})^{\top}$ and
$$\phi(x,t):=\sum_{j\in J^-}\alpha_j\ln (c_{j}^{\top}x+d_{j})
+\sum_{j\in J^+}\alpha_j \left( t_j(c_{j}^{\top}x+d_{j})-\ln t_j -1\right).$$

The equivalence between problems P$_1$ and P can be proved later. Before that, we give the following lemma for standby.

\begin{lemma}\label{lem1}
  For any positive real number $a\in (0,+\infty)$, the following inequality holds:
 \begin{equation}\label{eq4} \ln a +\frac{1}{a}-1\geq 0.
\end{equation}
 \end{lemma}
\begin{proof}
 Let $g(a)=\ln a +\frac{1}{a}-1$ for any $a\in (0,+\infty)$.  According to the first derivative of the univariate function, $g(a)$ is monotonically decreasing over the interval $(0,1]$ and monotonically increasing over the interval $[1,+\infty)$. Thus, it follows that
 $g(a)\geq g(1)=0$,  which shows that inequality (\ref{eq4}) holds.
  \qed
\end{proof}

\begin{theorem}\label{thm1}
     The point $x^*$  is globally optimal for problem {\rm P} if and only if $(x^*,t^*)$ is globally optimal for problem {\rm P}$_1$, where $t^*_j=\frac{1}{c_{j}^{\top}x^*+d_{j}}$, $j\in J^+$. \end{theorem}
\begin{proof}
 If $(x^*,t^*)$ is globally optimal for {\rm P}$_1$, $x^*$ is obviously feasible for the problem {\rm P}. Suppose that $x^*$ is not  optimal for the problem {\rm P},  there must be
a feasible solution $\hat{x}\in \mathcal{X}$ such that
 \begin{equation}\label{eq5}\nu(x)=\sum_{j=1}^{p}\alpha_j\ln (c_{j}^{\top}\bar{x}+d_{j})<\sum_{j=1}^{p}\alpha_j\ln (c_{j}^{\top}x^*+d_{j})=\nu(x^*).\end{equation}
Further, it follows from Lemma \ref{lem1} that
$$ t_j(c_{j}^{\top}x+d_{j})-\ln t_j-1-\ln (c_{j}^{\top}x+d_{j})=
\ln \frac{1}{t_j(c_{j}^{\top}x+d_{j})}+t_j(c_{j}^{\top}x+d_{j})-1\geq0
$$
hold for any $x\in \mathcal{X}$, $t_{j}>0$, $j\in J^+$. Hence, we have
 \begin{equation}\label{eq6} \sum_{j\in J^+} \alpha_j\ln (c_{j}^{\top}x+d_{j})\leq  \sum_{j\in J^+} \alpha_j \left( t_j(c_{j}^{\top}x+d_{j})-\ln t_j -1\right)  \end{equation}
 for any $x\in \mathcal{X}$, $t\in \mathcal{S}^0$, $j\in J^+$.
From the proof of Lemma \ref{lem1}, it can be concluded that the equation of Eq. (\ref{eq6}) holds if and only if $t_j=c_{j}^{\top}x+d_{j}$ for $j\in J^+$. Now, let $\hat{t}_j=\frac{1}{c_{j}^{\top}\hat{x}+d_{j}}$ for each $j\in J^+$, then $(\bar{x},\bar{t})$ is a feasible solution for {\rm P}$_1$. It holds from Eq. (\ref{eq6}) that
\begin{equation}\label{eq7}\nu(\hat{x})=\sum_{j=1}^p\alpha_j\ln (c_{j}^{\top}\hat{x}+d_{j})= \phi(\hat{x},\hat{t}) \end{equation}
Combining the Eq. (\ref{eq6}) and the feasibility of $(x^*,t^*)$ for {\rm P}$_1$, it can be derived that
\begin{equation}\label{eq8}\nu(x^*)=\sum_{j=1}^p\alpha_j\ln (c_{j}^{\top}x^*+d_{j})\leq\phi(x^*,t^*) \end{equation}
Then, from Eqs. (\ref{eq5}), (\ref{eq7}) and (\ref{eq8}),  we follow that
$\phi(\hat{x},\hat{t})<\phi(x^*,t^*)$, which contradicts the fact that $(x^*,t^*)$ is optimal for {\rm P}$_1$.
Moreover, the optimal solution $(x^*,t^*)$ of problem {\rm P}$_1$ must satisfy
\begin{equation}\label{eq9}t^*_j=\frac{1}{c_{j}^{\top}x^*+d_{j}}, j\in J^+. \end{equation}
 Otherwise, if there are some $j\in J^+$ such that $t^*_j\neq \frac{1}{c_{j}^{\top}x^*+d_{j}}$, let $\tilde{t}_j=\frac{1}{c_{j}^{\top}x^*+d_{j}}$ for $j\in J^+$. Obviously, $\tilde{t}\neq t^*$ and $(x^*,\tilde{t})$ is feasible for {\rm P}$_1$. It holds from Eq. (\ref{eq6}) that
 $$\phi(x^*,t^*)>\sum_{j=1}^p\alpha_j\ln (c_{j}^{\top}x^*+d_{j})=\phi(x^*,\tilde{t}).$$
This contradicts the
optimality of $(x^*,t^*)$ for {\rm P}$_1$.

Conversely, let $x^*$ be a  globally optimal solution to the problem {\rm P} and let $t^*_j=\frac{1}{c_{j}^{\top}x^*+d_{j}}$, $j\in J^+$. Thus $(x^*,t^*)$ is feasible for {\rm P}$_1$, and $\phi(x^*,t^*)=\nu(x^*)$. Suppose that $(x^*,t^*)$ is not optimal for {\rm P}$_1$, there must be a feasible solution $(\hat{x}, \hat{t})$ of {\rm P}$_1$ such that
\begin{equation}\label{eq10}
\phi(\hat{x},\hat{t})<\phi(x^*,t^*)=\nu(x^*). \end{equation}
From Eq. (\ref{eq6}), it holds that
\begin{equation}\label{eq12}\nu(\hat{x})\leq  \phi(\hat{x},\hat{t}),\end{equation}
Combining Eqs. (\ref{eq10}) and (\ref{eq12}), it follows that
$$\nu(\hat{x})<\nu(x^*),$$
which contradicts the
optimality of $x^*$ on {\rm P}.
 \qed
\end{proof}

Upon Theorem \ref{thm1}, we  reformulate {\rm P}$_1$ as the following problem EP: $$
({\rm EP}):\min\limits_{t\in  \mathcal{S}^0} ~\psi(t),$$
where $\psi(t)$ is the optimal value of the convex program with parameter $t$:
\begin{equation}\label{zeqj}
 \min\limits_{x\in \mathcal{X}} ~ \phi(x,t). \end{equation}

\begin{theorem}\label{thm2}
The point $t^*$ is  globally optimal for problem {\rm EP} if and only if  $(x^*,t^*)$ is globally optimal for problem {\rm P}$_1$ with $x^*\in\arg\min\limits_{x\in \mathcal{X}}\phi(x,t^*)$. Additionally, the global optimal values of problems {\rm P}$_1$  and {\rm EP} are equal.
\end{theorem}
\begin{proof}~Let $(x^*,t^*)$ be a global
optimal solution of problem {\rm P}$_1$, then  $x^*$  is obviously  feasible for the problem {\rm LP$_{t^*}$}  and $t^*$ is also  feasible for  the problem {\rm EP}.
Suppose that
 $t^*$ is not a global
optimal solution to problem {\rm EP}, there must be a feasible solution $\bar{t}$ to {\rm EP}  such that
 \begin{equation}\label{eq14}
\psi(\bar{t})<\psi(t^*).\end{equation}
Let $\bar{x}$ and $\hat{x}$ be the optimal solutions of problems {\rm LP$_{\bar{t}}$} and {\rm LP$_{t^*}$} respectively. Then~$(\bar{x},\bar{t})$ and
 ($\hat{x},r^*$) are both feasible solutions to problem {\rm P}$_1$.
Using Eq. (\ref{eq14}), there will be
$
\phi(\bar{x},\bar{t})=\psi(\bar{t})<\psi(t^*)
=\phi(\hat{x}, t^*)\leq \phi(x^*,t^*).$
This violates the optimality of $(x^*,t^*)$ for problem {\rm P}$_1$. Thus, $t^*$ is   globally  optimal for   {\rm EP}.

Conversely, let $t^*$ be a global
optimal solution to  {\rm EP},   we have
\begin{equation}\label{eq15}\psi(t^*)=\phi(x^*,t^*),~\forall x^*\in \arg\min\limits_{x\in \mathcal{X}}\phi(x,t^*) \end{equation}
and $(x^*,t^*)$ is  feasible for {\rm P}$_1$.  Suppose that  $(x^*,t^*)$ is not a global  optimal solution to {\rm P}$_1$,   there is another feasible solution $(\bar{x},\bar{t})$ to {\rm P}$_1$ that satisfies
 \begin{equation}\label{eq16}\bar{x}\in  \mathcal{X},~
\phi(\bar{x},\bar{t})<\phi(x^*,t^*)=\psi(t^*).\end{equation}
Let $\hat{x}\in  \arg\min\limits_{x\in  \mathcal{X}}\phi(x,\bar{t})$,
then we have
 \begin{equation}\label{eq17}\phi(\bar{x},\bar{t})\geq \phi(\hat{x},\bar{t})=\psi(\bar{t}).\end{equation}
Notice that $\bar{t}\in \mathcal{H}^0$,  it can be inferred from formulas (\ref{eq16})-(\ref{eq17}) that $\psi(\bar{t})<\psi(t^*)$, which contradicts the optimality of $t^*$ on the problem {\rm EP}. Thus, $(x^*,t^*)$ is a global optimal solution of problem {\rm P}$_1$.

From the above discussion, if  $t^*$ is  globally optimal for problem {\rm EP},  $(x^*,t^*)$ must be a global optimal solution of problem {\rm P}$_1$ that satisfies Eq. (\ref{eq15}).  Hence, problems {\rm P}$_1$ and {\rm EP} have the same global optimal value. \qed
\end{proof}

Theorems \ref{thm1}-\ref{thm2} state that problems {\rm P}, {\rm P}$_1$ and {\rm EP} are equivalent, and their global optimal values are equal. It can be known that the process of dealing with the problem {\rm P} or {\rm P}$_1$  can be realized by solving the {\rm EP} problem.
 In view of this characteristic, we will propose a class of nonconvex relaxations for the problem EP and its subproblems in the next section.

\section{Novel convex relaxation approach}\label{SEZ3}
The main purpose of this section is to construct a simplex based convex relaxation strategy by utilizing the {\rm EP} problem.


Let $\mathcal{S}:=[v^1,v^2,\cdots,v^{\bar{p}+1}]$  be $\mathcal{S}^0$ or its any  sub-simplex, and  denote the two sets of vertices and edges of $\mathcal{S}$ by $$V(\mathcal{S}) = \{v^1,\cdots,v^{\bar{p}+1}\},~~E(\mathcal{S}) = \{\{v^i, v^j\} : i, j\in J^+\cup \{\bar{p}+1\},i<j\}.$$
Define $d(\mathcal{S})$ as the length of the longest edge of  $\mathcal{S}$ by $d(\mathcal{S}) = \max\{\|v-w\|:\{v, w\}\in E(\mathcal{S})\}$.

 Then the subproblem of the problem EP over $\mathcal{S}$ can be expressed as
$$
({\rm EP_{\mathcal{S}}}):\min\limits_{t\in\mathcal{S}} ~\psi(t), $$
where $\psi(t)=\min\limits_{x\in \mathcal{X}} ~ \phi(x,t)$.

 Since  $\psi(t)$ is non-convex with respect to $t$,   we will propose a  lower bounding function for underestimating  $\psi(t)$ over $\mathcal{S}$. To this end,
we  define
 $$\triangle_{\bar{p}}=\left\{w\in \mathbb{ R}^{\bar{p}+1}\bigg| \sum_{i=1}^{\bar{p}+1}w_i=1,
w_i\geq0,i=1,\cdots,\bar{p}+1\right\}.$$

Next, the specific lower bounding method is provided by the following theorem.

\begin{theorem}\label{jthe1}
 The following convex program provides an lower bound for {\rm EP$_{\mathcal{S}}$}:
\begin{equation}\label{eqj5} LB(\mathcal{S})=\min\limits_{w\in \triangle_{\bar{p}}} \sum_{i=1}^{\bar{p}+1}\left(\psi(v^i)+\sum_{j=1}^{\bar{p}}\alpha_j\ln v^i_j\right)w_i-\sum_{j=1}^{\bar{p}}\alpha_j\ln \left(\sum_{i=1}^{\bar{p}+1}v_j^iw_i\right).\end{equation}
 \end{theorem}

\begin{proof}
 For $\psi(t)$ defined in {\rm EP$_{\mathcal{S}}$}, after relaxing $\sum_{j=\bar{p}+1}^p\alpha_j\ln (c_{j}^{\top}x+d_{j}) -\sum_{j=1}^{\bar{p}}\alpha_j$ and $\alpha_j(c_{j}^{\top}x+d_{j})$ into $y_0\in \mathbb{R}$ and $y_{j}\in \mathbb{R}$, $j\in J^+$, respectively, it follows that
 \begin{equation}\label{eqj4}\begin{aligned}
 \psi(t)&\geq \underline{\psi}(t):=\min y_0+ t^{\top}y-\sum_{j=1}^{\bar{p}}\alpha_j\ln t_j~~~~~~~~~~~~~~~~~~~~~~~~~~~~~~~~~~~~~~~~~~~~~~~~~~~~~~~~~~~~~~~~~~~~~~~~~~~~~~~~~~~~\\
 &~~~~~~~~~~~~~~{\rm s.t.}~~y_0+(v^i)^{\top}y\geq \psi(v^i)+\sum_{j=1}^{\bar{p}}\alpha_j\ln v^i_j,~i\in J^+\cup \{p+1\},\\
 &~~~~~~~~~~=\max\limits_{\mu\in \triangle_{\bar{p}}} \sum_{i=1}^{\bar{p}+1}\mu_i\left(\psi(v^i)+\sum_{j=1}^{\bar{p}}\alpha_j\ln v^i_j\right)-\sum_{j=1}^{\bar{p}}\alpha_j\ln t_j\\
  &~~~~~~~~~~~~~~{\rm s.t.}~~\sum_{i=1}^{\bar{p}+1}\mu_iv^i=t,
\end{aligned}
\end{equation}
 where the right-hand side of the second equality is the Lagrangian dual of the  former linear minimization problem and the gap between them is zero.

From $V(\mathcal{S})=\{v^1,\cdots,v^{\bar{p}+1}\}$, we know $\mathcal{S}=\left\{\sum_{i=1}^{\bar{p}+1}\mu_iv^i|\mu\in \triangle_{\bar{p}}\right\}$. Then for any $t\in \mathcal{S}$, there always exists a $w\in \triangle_{p}$ such that $t=\sum_{i=1}^{\bar{p}+1}w_iv^i$, and the last maximization problem in Eq. (\ref{eqj4})  has a unique solution $\mu=w$ when $\mathcal{S}$ is non-degenerate. This fact implies that
 \begin{equation}\label{eqj6}\min\limits_{t\in \mathcal{S}}\psi(t)\geq \min\limits_{t\in \mathcal{S}}\underline{\psi}(t)=LB(\mathcal{S}), \end{equation}
where $LB(\mathcal{S})$ is defined in  (\ref{eqj5}). This completes the proof.
\end{proof}

Theorem \ref{jthe1} reveals that convex optimization problem (\ref{eqj5}) is able to provide a valid lower bound for {\rm EP$_{\mathcal{S}}$}. Also, $\min\{\psi(v^{i}):i=1,2,\cdots,\bar{p}+1\}$ may be able to update the upper bound on the global optimum of problem EP. In effect, $LB(\mathcal{S})$ is compact enough if the the length $d(\mathcal{S})$ of the longest edge becomes sufficiently small, i.e.,
\begin{equation}\label{eqjq9} \lim_{d(\mathcal{S})\rightarrow0}\left(\min\limits_{t\in \mathcal{S}}\psi(t)-LB(\mathcal{S})\right)\rightarrow0,\end{equation}
 which will be proved in the subsequent Theorem \ref{jthe4}.

\section{A new branch and bound algorithm}\label{SEZ4}
\subsection{Simplicial  bisection method}\label{ssub4.1}
Derived from the fact that the above convex relaxation is proposed based on the simplex, the branching operation of our algorithm will be performed on the simplex. Thus, the simplex $\mathcal{S}_0$ will be subdivided into many sub-simplices.
For a given simplex $\mathcal{S}^k\subseteq \mathcal{S}^0$ with $V(\mathcal{S}^k) = \{v^1_k,\cdots,v^{\bar{p}+1}_k\}$ and $E(\mathcal{S}^k) = \{\{v^i_k, v^j_k\} : i, j =1,\cdots, \bar{p}+1,i<j\}$, we now give the branching rule as follows:

i) Choose a $\{v^{\tilde{i}}_k, v^{\tilde{j}}_k\}\in\arg\max\{\|v-w\||\{v, w\}\in E(\mathcal{S}^k)\}$ and set $\eta=\frac{v^{\tilde{i}}_k+v^{\tilde{j}}_k}{2}$;

ii) Divide $\mathcal{S}^{k}$
into
$$
\mathcal{S}^{k1}=
\left\{\sum_{l=1,l\neq \tilde{j} }^{\bar{p}+1}w_lv^l_k+w_{\tilde{j}}\eta\bigg|w\in \triangle_{p}\right\}~
{\rm and}~
\mathcal{S}^{k2}=
\left\{\sum_{l=1,l\neq \tilde{i}}^{\bar{p}+1}w_lv^l_k+w_{\tilde{i}}\eta \bigg|w\in \triangle_{p}\right\},
$$
where  $V(\mathcal{S}_{k1}) = \{v^1_{k1},\cdots,v^{\bar{p}+1}_{k1}\}$ with $v^{\tilde{j}}_{k1}=\eta$,
$v^l_{k1}=v^l_{k}$ for $l\in J^+\cup\{\bar{p}+1\}\backslash\{\tilde{j}\}$,  $V(\mathcal{S}_{k2}) = \{v^1_{k2},\cdots,v^{\bar{p}+1}_{k2}\}$ with $v^{\tilde{i}}_{k2}=\eta$,
$v^l_{k2}=v^l_{k}$ for $l\in J^+\cup\{\bar{p}+1\}\backslash\{\tilde{i}\}$.


From the above simplex partition rule, we know that $\mathcal{S}^{k1}\cap \mathcal{S}^{k2}={\rm rbd}(\mathcal{S}^{k1})\cap {\rm rbd}(\mathcal{S}^{k2})$, $\mathcal{S}^{k1}\cup \mathcal{S}^{k2}=\mathcal{S}^{k}$ and $\mathcal{S}^{ks}\subset \mathcal{S}^{k},s=1,2$, where ${\rm rbd}(\mathcal{S}^{k1})$ denotes the relative boundary
of  $\mathcal{S}^{k1}$.
Besides, this simplex partition rule follows the longest edge partitioning rule, it generates an exhaustive nested partition of a given simplex in the limit  sense \cite{Dickinsonhapter3}. Thus, the global convergence of the B\&B algorithm is guaranteed.
\subsection{Simplicial branch-and-bound algorithm}
To find the global optimal solution to the lifted problem EP, we construct a new outer-space simplicial branch and bound algorithm (OSSBBA) by adding the proposed convex relaxation technique and simplex branching rules to the B\&B framework.
\begin{algorithm}[p!]
\textbf{Algorithm}~(OSSBBA)\\
\textbf{Step 0 (Initialization).}

Given a  tolerance $\epsilon>0$. Calculate $\underline{t}_j$ and $\overline{t}_j$, $j=1,2,\cdots,\bar{p}$.

 Initialize the simplex $\mathcal{S}^0$ with a set  $V(\mathcal{S}^0)=\{v^{0,1},\cdots,v^{0,\bar{p}+1}\}$ of vertices, where $v^{0,1}=\underline{t}$, $v^{0,j+1}=(\underline{t}_1,\cdots,\underline{t}_j+\bar{p}(\overline{t}_j-\underline{t}_j),\cdots,\underline{t}_{\bar{p}})$, $i=1,2,\cdots,\bar{p}$.
  Clearly, $\mathcal{S}^0$ covers $[\underline{t},\overline{t}]$.

Compute $t^*=\arg\min\{\psi(t^{0,i}):i=1,2,\cdots,\bar{p}+1\}$ and let $UB^0=\psi(t^*)$  be the initial upper bound on the optimal value of problem EP.

Solve the convex  problem (\ref{eqj5}) to obtain an initial lower bound $LB^0=LB(\mathcal{S}^0)$ on the optimal value of EP.

 Set $\Xi=\{[\mathcal{S}^0, E(\mathcal{S}^0),LB(\mathcal{S}^0)]\}$,
   $k:=0$.\\
\textbf{Step 1 (Termination).}

If $UB^k-LB^k\leq \epsilon$, terminate and output $t^*$.\\
\textbf{Step 2. (Simplicial subdivision).}

Choose a $\{v^{\tilde{i}}_k, v^{\tilde{j}}_k\}\in\arg\max\{\|v-w\|:\{v, w\}\in E(\mathcal{S}^k)\}$ and set $\eta=\frac{v^{\tilde{i}}_k+v^{\tilde{j}}_k}{2}$.

By using the simplicial bisection method shown in Sect. \ref{ssub4.1}, the simplex $S_k$ is divided into two sub-simplices $\mathcal{S}^{k1}$ and $\mathcal{S}^{k2}$ with $V(\mathcal{S}^{k1})=\{v^1_{k1},\cdots,v^{r+1}_{k1}\}$  and $V(\mathcal{S}^{k2})=\{v^1_{k2},\cdots,v^{\bar{p}+1}_{k2}\}$, where
 $v^{\tilde{j}}_{k1}=\eta$, $v^j_{k1}=v^j_{k}$, $j\in J^+\backslash\{\tilde{j}\}$,
$v^{\tilde{i}}_{k2}=\eta$, $v^i_{k2}=v^i_{k}$, $i\in J^+\backslash\{\tilde{i}\}$.
\\
\textbf{Step 3. (Determine the upper bound).}

Compute $U^k=\psi(\eta)$, if $U^k<UB^k$, set $UB^k=U^k$, $t^*=\eta$.\\
\textbf{Step 4. (Simplicial delete).}

Compute $LB(\mathcal{S}^{k1})$ and $LB(\mathcal{S}^{k2})$.

Set~$\Xi:=\Xi\backslash \{[\mathcal{S}^k, E(\mathcal{S}^k),LB(\mathcal{S}^k)]\} \cup \{[\mathcal{S}^{k1}, E(\mathcal{S}^{k1}),LB(\mathcal{S}^{k1})],[\mathcal{S}^{k2}, E(\mathcal{S}^{k2}),LB(\mathcal{S}^{k2})]\}$.

For each $s=1,2$, if $UB^k-LB(\mathcal{S}^{ks})\leq\epsilon$, set $\Xi:=\Xi\backslash  \{[\mathcal{S}^{ks}, E(\mathcal{S}^{ks}),LB(\mathcal{S}^{ks})]\}$.\\
\textbf{Step 5. (Determine the lower bound).}

If $\Xi\neq\emptyset$,
set $LB^{k}:=\min\{LB(\mathcal{S}):[\cdot,\cdot,LB(\mathcal{S})]\in\Xi\}
$ and goto Step 6, otherwise stop and output $t^*$, $k=k+1$.\\
\textbf{Step 6 (Select the simplex).}

Choose an element $\{[\mathcal{S}^{k},\cdot,\cdot]\}\in  \Xi$ such that  $LB(\mathcal{S}^{k})=LB^k$.

Set $k:=k+1$ and return to Step 1.
\end{algorithm}

In the above algorithm, $k$ is adopted as the iteration index. At each iteration, one subproblem is selected and up to two new subproblems are created to replace the old one. The optimal value of the new subproblem will gradually improve compared with the old subproblem, and the corresponding upper and lower bounds will be updated, so that the gap between the upper and lower bounds will gradually decrease.
Indeed, for a given tolerance $\epsilon>0$, OSSBBA can output a global $\epsilon$-optimal solution for problem EP. Here, we call $t^*\in S_0$ a global $\epsilon$-optimal solution to problem EP  if $\psi(t^*)\leq \min\limits_{t\in \mathcal{S}^0}\psi(t)+\epsilon$.

\begin{theorem}\label{jthe4}
Given a tolerance $\epsilon>0$, when {\rm OSSBBA} runs to Step 1 of the $k$th iteration, if the subproblem $\{[\mathcal{S}^k,E(\mathcal{S}^k),LB(\mathcal{S}^k)]\}$ satisfies $d(\mathcal{S}^k)\leq \frac{\epsilon}{\sum_{j=1}^{\bar{p}}\frac{\alpha_j}{d_{kj}}}$ with $d_{kj}=\min\limits_{t\in \mathcal{S}^k}t_j$, the algorithm must terminate and output a global $\epsilon$-optimal solution to problem {\rm EP}. Besides, $ \min\limits_{t\in \mathcal{S}^k}\psi(t)-LB(\mathcal{S}^k)\leq\epsilon.$
 \end{theorem}
\begin{proof}
Since $LB^k=LB(\mathcal{S}^k)$ is the least lower bound at the current iteration, we have $UB^k\geq \min\limits_{t\in \mathcal{S}^k}\psi(t)\geq LB^k$.
Let $t^{\bar{i}}_k\in\arg\min\{\sum_{j=1}^{\bar{p}}\alpha_j\ln v^i_{kj}|i=1,\cdots,\bar{p}+1\}$ and $\tilde{\psi}_k=\min\{\psi(v_k^i)|i=1,\cdots,\bar{p}+1\}$,
 and then from Eq. (\ref{eqj6}) we have

 \begin{equation}\label{eqj7}\tilde{\psi}_k+\sum_{j=1}^{\bar{p}}\alpha_j\ln t^{\bar{i}}_{kj}-\sum_{j=1}^{\bar{p}}\alpha_j\ln \bar{t}_{kj}\leq LB(\mathcal{S}^k)=LB^k\leq\min\limits_{t\in \mathcal{S}^k}\psi(t)\leq UB^k\leq \tilde{\psi}_k, \end{equation}
 where  $\bar{t}_k\in S^k$ is the optimal solution of the  problem $ \min\limits_{t\in \mathcal{S}^k}\underline{\psi}(t)$.
For all $t\in \mathcal{S}^k$, for each $j\in J^+$, it holds  that
\begin{equation}\label{eqj9} \ln t^{\bar{i}}_{kj}-\ln t_j\leq\frac{1}{\xi_{j}}|t^{\bar{i}}_{kj}-t_j|\leq \frac{1}{d_{kj}}|t^{\bar{i}}_{kj}-t_j|\leq \frac{1}{d_{kj}}d(\mathcal{S}^k),\end{equation}
where $\xi_{j}\in \mathbb{R}$ satisfies the  mean value theorem of the logarithmic function $\ln z$ over the interval $\left[\min\{t^{\bar{i}}_{kj},t_j\},\max\{t^{\bar{i}}_{kj}
,t_j\}\right]$. It follows from Eqs.  (\ref{eqj7})-(\ref{eqj9}) that
\begin{equation}\label{eqj1410}\min\limits_{t\in \mathcal{S}^k}\psi(t)-LB(\mathcal{S}^k)\leq UB^k-LB^k\leq \sum_{j=1}^{\bar{p}}\alpha_j(\ln t^{\bar{i}}_{kj}-\ln \bar{t}_{kj})\leq \sum_{j=1}^{\bar{p}}\frac{\alpha_j}{d_{kj}}d(\mathcal{S}^k).\end{equation}
Thus, when $d(\mathcal{S}^k)\leq \frac{\epsilon}{\sum_{j=1}^{\bar{p}}\frac{\alpha_j}{d_{kj}}}$ is satisfied, we have
\begin{equation}\label{eqj110} \min\limits_{t\in \mathcal{S}^k}\psi(t)-LB(\mathcal{S}^k)\leq UB^k-LB^k\leq\epsilon.\end{equation}
Besides,  it knows from the iterative mechanism of OSSBBA that $LB^k\leq\min\limits_{t\in \mathcal{S}^0}\psi(t)\leq UB^k=\psi(t^*)$.
Then by combining Eq. (\ref{eqj110}), it follows that
$$\psi(t^*)= UB^k\leq  LB^k+\epsilon\leq \min\limits_{t\in \mathcal{S}^0}\psi(t)+\epsilon,$$which means that $t^*$ is a global $\epsilon$-optimal solution to EP.
\end{proof}
\begin{remark}\label{REK1}From Theorem \ref{jthe1},  we know $\min\limits_{t\in \mathcal{S}^k}\psi(t)-LB(\mathcal{S}^k)\geq0$, thus
Eq. (\ref{eqj1410}) implies that Eq. (\ref{eqjq9}) holds.\end{remark}

 Theorem \ref{jthe4} and Remark \ref{REK1} show  that the optimal value of problem EP over each subsimplex and that of its relaxation problem (\ref{eqj5}) are gradually approaching in the limiting sense. This  implies that the bounding and branching operations are consistent, so the B\&B algorithm is theoretically globally convergent.


Now, let us analyze the complexity of OSSBBA based on Theorem \ref{jthe4} and the iterative mechanism of this algorithm.

\begin{theorem} \label{jthe5}
Given a tolerance $\epsilon>0$, the maximum number of iterations required by {\rm OSSBBA} to obtain a global $\epsilon$-optimal solution for {\rm GLMP} is
$$\left\lfloor\frac{\prod_{i=1}^{\bar{p}}(\overline{t}_i-\underline{t}_i)}{\sqrt{\bar{p}+1}}
\left(\frac{\sqrt{2}\bar{p}\sum_{j=1}^{\bar{p}}\frac{\alpha_j}
{\underline{t}_j}}{\epsilon}\right)^{\bar{p}}\right\rfloor.$$
\end{theorem}
\begin{proof}
When OSSBBA terminates, either $k=0$ or $k\geq1$. If $k=0$, the algorithm does not enter the iteration loop.
Thus,
let us talk about the case where the algorithm terminates after many iterations.

At the case of $k\geq1$, it follows from Theorem \ref{jthe4} that $d(\mathcal{S}^k)\leq \frac{\epsilon}{\sum_{j=1}^{\bar{p}}\frac{\alpha_j}{d_{kj}}}$ is essentially a  sufficient condition for the termination criterion of the algorithm, $UB^k-LB^k\leq\epsilon$, to hold.
Since $S_k\subseteq \mathcal{S}^0$, it follows   that $d_{kj}=\min\limits_{t\in \mathcal{S}^k}t_j \geq \underline{t}_j>0$  for all $j\in J^+$.
Hence,  $\frac{\epsilon}{\sum_{j=1}^{\bar{p}}\frac{\alpha_j}{d_{kj}}}\geq
\frac{\epsilon}{\sum_{j=1}^{\bar{p}}\frac{\alpha_j}{\underline{t}_j}}$,
which means $d(\mathcal{S}^k)\leq \frac{\epsilon}{\sum_{j=1}^{\bar{p}}\frac{\alpha_j}{\underline{t}_j}}$ is also a sufficient condition for $UB^k-LB^k\leq\epsilon$.
Further, according to the simplicial branching rule in Step 2, a total of $k+1$ sub-simplices are generated for the initial simplex $\mathcal{S}^0$.
For convenience, we denote these subsimplices as $\mathcal{S}^1,\mathcal{S}^2,\cdots,\mathcal{S}^{k+1}$, respectively.
Obviously, $\mathcal{S}^0=\bigcup\limits_{t=1}^{k+1}\mathcal{S}^t$.
Also, let $V(\mathcal{S}^t) = \{v_t^1,\cdots,v_t^{r+1}\}$, $E(\mathcal{S}^t) = \{\{v_t^i, v_t^j\} : i, j =1,\cdots, \bar{p}+1,i<j\}$ for $t=1,\cdots, k+1$.
 In the worst case, suppose that the longest edge of each subsimplex $S_t$ satisfies $d(\mathcal{S}^t)\leq \frac{\epsilon}{\sum_{j=1}^{\bar{p}}\frac{\alpha_j}{\underline{t}_j}}$. At this point, every edge $\{v_t^i, v_t^j\}$ of $\mathcal{S}^t$ satisfies
\begin{equation}\label{EQ12}\|v_t^i-v_t^j\| \leq \frac{\epsilon}{\sum_{j=1}^{\bar{p}}\frac{\alpha_j}{\underline{t}_j}},~i, j =1,\cdots, \bar{p}+1,i<j,t=1,\cdots, k+1,\end{equation}
which implies that the volume $\mathcal{V}(\mathcal{S}^t)$ of $\mathcal{S}^t$ does not exceed the volume $\mathcal{V}(\bar{\mathcal{S}})$ of a $\bar{p}$-simplex $\bar{\mathcal{S}}$  with a unique edge length $\frac{\epsilon}{\sum_{j=1}^{\bar{p}}\frac{\alpha_j}{\underline{t}_j}}$.
Thus, we have
\begin{equation}\label{EQ13}\mathcal{V}(\mathcal{S}^0)=\sum\limits_{t=1}^{k+1}\mathcal{V}(\mathcal{S}^t)\leq (k+1)\mathcal{V}(\bar{\mathcal{S}}).\end{equation}
By Cayley-Menger Determinant, it can be calculated that
\begin{equation}\label{EQ14}\begin{aligned}&\mathcal{V}(\mathcal{S}^0)=\sqrt{\frac{(-1)^{\bar{p}+1}}{2^{\bar{p}}(\bar{p}!)^2}
\det\left(\left|\begin{matrix}
0 & e_{\bar{p}+1}^{\top} \\
e_{\bar{p}+1} & \hat{B} \\
\end{matrix}\right|\right)}=\frac{\bar{p}^{\bar{p}}\prod_{i=1}^{\bar{p}}(\overline{t}_i-\underline{t}_i)}{\bar{p}!},\\ &\mathcal{V}(\bar{\mathcal{S}})=\sqrt{\frac{(-1)^{\bar{p}+1}}{2^{\bar{p}}(\bar{p}!)^2}\det\left(\left|\begin{matrix}
0 & e_{\bar{p}+1}^{\top} \\
e_{\bar{p}+1} & \check{B} \\
\end{matrix}\right|\right)}=
\frac{\sqrt{\bar{p}+1}}{\bar{p}!}\left(\frac{\epsilon}{\sqrt{2}\sum_{j=1}^{\bar{p}}\frac{\alpha_j}{\underline{t}_j}}\right)^{\bar{p}},\end{aligned}\end{equation} where   $\hat{B}=(\hat{\beta}_{ij})$ and $\check{B}=(\check{\beta}_{ij})$ denote two $(\bar{p}+1)\times(\bar{p}+1)$ matrixes given by $\hat{\beta}_{ij}=\|v^{0,i}-v^{0,j}\|^2$, $\check{\beta}_{ij}=\left(\frac{\epsilon}{\sum_{j=1}^{\bar{p}}\frac{\alpha_j}{\underline{t}_j}}\right)^2$ for
$i, j =1,\cdots, \bar{p}+1, i\neq j$, and $\hat{\beta}_{ii}=\check{\beta}_{ii}=0$ for $i =1,\cdots, \bar{p}+1$. From the properties of the vertices of $S_0$ given in Step 0 of OSSBBA, it follows that $\hat{\beta}_{1j}=\hat{\beta}_{j1}=\bar{p}^2(\overline{t}_{j-1}-\underline{t}_{j-1})^2$ for $j=2,3,\cdots, \bar{p}+1$ and $\hat{\beta}_{ij}=\hat{\beta}_{ji}=\bar{p}^2[(\overline{t}_{i-1}-\underline{t}_{i-1})^2+(\overline{t}_{j-1}-\underline{t}_{j-1})^2]$ for $i, j =2,3,\cdots, \bar{p}+1, i\neq j$. Next, by combining Eqs. (\ref{EQ13}) and (\ref{EQ14}), we have
$$k\geq \frac{\mathcal{V}(\mathcal{S}^0)}{\mathcal{V}(\bar{\mathcal{S}})}-1=\frac{\prod_{i=1}^{\bar{p}}(\overline{t}_i-\underline{t}_i)}{\sqrt{\bar{p}+1}}
\left(\frac{\sqrt{2}\bar{p}\sum_{j=1}^{\bar{p}}\frac{\alpha_j}{\underline{t}_j}}{\epsilon}\right)^{\bar{p}}-1.$$
However, when Eq. (\ref{EQ12}) holds and $k=\left\lfloor\frac{\prod_{i=1}^{\bar{p}}(\overline{t}_i-\underline{t}_i)}{\sqrt{\bar{p}+1}}\left(\frac{\sqrt{2}\bar{p}\sum_{j=1}^{\bar{p}}\frac{\alpha_j}
{\underline{t}_j}}{\epsilon}\right)^{\bar{p}}\right\rfloor$, these $k+1$ subsimplices must be deleted in Step 4.
Thus, the number of iterations at which OSSBBA terminates is at most $\left\lfloor\frac{\prod_{i=1}^{\bar{p}}(\overline{t}_i-\underline{t}_i)}{\sqrt{\bar{p}+1}}
\left(\frac{\sqrt{2}\bar{p}\sum_{j=1}^{\bar{p}}\frac{\alpha_j}
{\underline{t}_j}}{\epsilon}\right)^{\bar{p}}\right\rfloor$. This completes the proof.

 \end{proof}
\begin{remark}
Theorem  \ref{jthe5} reveals that when OSSBBA finds a global $\epsilon$-optimal solution for GLMP, the computational time required is at most $$(T_n+2T_{\bar{p}})\left\lfloor\frac{\prod_{i=1}^{\bar{p}}(\overline{t}_i-\underline{t}_i)}{\sqrt{\bar{p}+1}}
\left(\frac{\sqrt{2}\bar{p}\sum_{j=1}^{\bar{p}}\frac{\alpha_j}{\underline{t}_j}}
{\epsilon}\right)^{\bar{p}}\right\rfloor+(\bar{p}+1)T_n$$  seconds, where $T_n$ and $T_r$ denote the upper bounds of the time required to compute a  function value $\psi(t)$ and solve a convex problem (\ref{eqj5}), respectively.
\end{remark}
\begin{remark}
Theorem  \ref{jthe5} sufficiently guarantees that the algorithm OSSBBA completes termination in a finite number of iterations because of the existence of this most extreme number of iterations.
\end{remark}
\section{Numerical Experiments}\label{SEZ5}
 To verify the  effectiveness and feasibility of the proposed algorithm, the code was compiled on Matlab (2017a) and a series of numerical experiments were performed.  All calculations were carried out on a desktop computer with Windows 7 operating system, Intel(R) Core(TM)i5-8500 3.00 GHz power processor 8 GB memory. Besides, all linear programs are addressed by linprog solver.

\subsection{Feasibility tests}
To demonstrate the effectiveness of our algorithm, we dealt with a total of nine small-size instances from \cite{25B} and then compared the results to the algorithms in \cite{5B,25B} and  a popular commercial solver BARON \cite{Baron}.   Table \ref{TB1} displays the experimental results. The symbols in the header line of Table \ref{TB1} are translated as follows:  Time: the CPU time in seconds; Iter: the number of iterations; Solution: the optimal solution; Optimum: the optimal value; $\epsilon$: the tolerance.

\begin{table}[p!]\scriptsize
\caption{Comparison of results for Examples 1-9\label{TB1}}
\centering
\begin{tabular}{clllllllcccccccccc}
\toprule
E  &  Refs 	&  Solution    &  Optimum     & Iter    &  Time  \\
\midrule
1 
&\cite{5B}  &(2.0, 8.0)&10.0&2  &0.05  \\
&\cite{25B}  &(2.0, 8.0)&10.0&0  &0.01  \\
&BARON  &(2.0, 8.0)&10.0&1  &0.05 \\
&ours  &(2.0, 8.0)&10.0&0  &0.04  \\
\midrule
2
 &\cite{5B} &(1.3148, 0.1396, 0.0, 0.4233) & 0.8902  &1  &0.03   \\
  &\cite{25B} &(1.3148, 0.1396, 0.0, 0.4233) & 0.8902  &1  &0.03   \\
 &BARON &(1.3148, 0.1396, 0.0, 0.4233) & 0.8902  &7  &0.05  \\
&ours &(1.3148, 0.1396, 0.0, 0.4233) & 0.8902  &1  &0.05   \\
\midrule
3
   &\cite{5B}   &(8.0, 0.0, 1.0)  &0.901235    &9 &0.28  \\
    &\cite{25B}   &(8.0, 0.0, 1.0)  &0.901235    &8 &0.23  \\
    &BARON&(8.0, 0.0, 1.0) &0.901235 &9&0.06    \\	
  &ours   &(8.0, 0.0, 1.0)  &0.901235    &117 &5.12 \\
\midrule

4   
 &\cite{5B}&(1.0, 2.0, 1.0, 1.0, 1.0)  &9504.0&3  &0.09\\
 &\cite{25B}&(1.0, 2.0, 1.0, 1.0, 1.0)  &9504.0&0  &0.01\\
  &BARON&(1.0, 2.0, 1.0, 1.0, 1.0)  &9504.0&1&0.06 \\
 &ours&(1.0, 2.0, 1.0, 1.0, 1.0)  &9504.0&6  &0.37 \\
\midrule
5  
  &\cite{5B}&(1, 1)  &997.6613&0   & 0.01       \\
 & \cite{25B}&(1, 1)  &997.6613&0   & 0.01      \\
   &BARON&(1, 1)  &997.6613&1   & 0.02       \\
 &ours&(1, 1)  &997.6613&0   & 0.03       \\
\midrule
6 
  &\cite{5B}&(1.25, 1.0)  &263.7889&1   &0.02       \\
   &\cite{25B}&(1.25, 1.0)  &263.7889&1   &0.02       \\
   &BARON&(1.25, 1.0)  &263.7889&1  &0.06   \\
   &ours&(1.25, 1.0)  &263.7889&0   &0.05       \\
\midrule
7
 &\cite{5B}   &(0.0, 0.0) & 0.53333  &71  &2.15      \\
 &\cite{25B}   &(0.0, 0.0) & 0.53333  &55  &1.65      \\

  &BARON   &(0.0, 0.0) & 0.53333  &1  &0.03    \\

 &ours   &(0.0, 0.0) & 0.53333  &5  &1.72    \\

\midrule

8 
  &\cite{5B} &(1.5, 1.5)  &5.0987&26 &0.77     \\
 &\cite{25B}  &(1.5, 1.5)  &5.0987&24 &0.74      \\

  &BARON &(1.5, 1.5)  &5.0987&3  &0.04     \\
 &ours&(1.5, 1.5)  &5.0987&12  &3.49      \\
\midrule
9  
  &\cite{5B} &(1.8000, 2.0000)  &0.2101&108 &2.95    \\
  &\cite{25B} &(1.8000, 2.0000)  &0.2101&56 & 1.53     \\
   &BARON &(1.8000, 2.0000)  &0.2101&1 &0.05    \\
  &ours &(1.8000, 2.0000)  &0.2101&12 &4.86    \\
\bottomrule
\end{tabular}
\end{table}

 From Table  \ref{TB1}, it is clear that although it is successful in handling all of the problems within a constrained CPU time and number of iterations, our algorithm is not particularly advantageous in solving these nine small-scale problems. The optimal solutions found are comparable to those provided by the BARON software as well as the other two algorithms. Thus, the feasibility of our algorithm can be preliminarily verified.


\subsection{Computational efficiency test}
In order to assess the computational efficacy of Algorithm OSSBBA, we are currently demonstrating three randomized generation schemes for the GLMP problem.
$$
({\rm P1})~:\left\{\begin{aligned}&\min~~\prod_{j=1}^{2}(c_{j}^{\top}x+1)\\
&~{\rm s.t.}~~\sum_{i=1}^nA_{si}x_i\leq b_s, \; s=1,2,\cdots,m,\\
&~~~~~~~~~ x_i\geq 0, \; i=1,2,\cdots,n,
\end{aligned}\right.$$
$$
({\rm P2})~:\left\{\begin{aligned}&\min~~\prod_{j=1}^{p}c_{j}^Tx\\
&~{\rm s.t.,}~~\sum_{i=1}^nA_{si}x_i\leq b_s, \; s=1,2,\cdots,m,\\
&~~~~~~0\leq x_i\leq 1, \; i=1,2,\cdots,n,
\end{aligned}\right.$$
 $$
({\rm P3})~:\left\{\begin{aligned}&\min~~\prod_{j=1}^{p}(c_{j}^Tx+d_{j})^{\alpha_{j}}\\
&~{\rm s.t.,}~~\sum_{i=1}^nA_{si}x_i\leq b_s, \; s=1,2,\cdots,m,\\
&~~~~~~ x_i\geq 0, \; i=1,2,\cdots,n,
\end{aligned}\right.$$
where $A_{si},\alpha_{j}$ is randomly generated in the interval [-1,1], and the value of the right $b_s$ is generated by $\sum_{i=1}^nA_{si}+2\mu$, where $\mu$ is randomly generated in the interval [0,1].  Also, $c_{j},d_{j}$ are randomly generated in  [0,1].

 For each of the random problems mentioned above, we generate ten instances at various scales randomly. Tables \ref{TB2}-\ref{TB8} summarize the average numerical results obtained by the relevant algorithms on these instances. In these ten tables, the symbols have the following explanations: Avg.Time represents the average CPU time, Avg.Iter represents the average number of iterations, and Opt.val represents the average of the optimal values. The symbol "$-$" indicates that none of the ten test instances could be solved within 3600s.  For the problems P1-P3, we not only compare Algorithm OSSBBA with BARON, but also compare it with the two algorithms recently proposed by Jiao et al. \cite{5B,25B}.

\begin{table}[h!]\scriptsize
\caption{Numerical comparison between  OSSBBA and BARON on    P1  \label{TB2}}
\centering
\begin{tabular}{lllllllllllllll}
\toprule
\multirow{2}*{\textbf{$(m,n)$}}&\multicolumn{3}{c}{ OSSBBA}&
  &\multicolumn{3}{c}{BARON}\\
 \cline{ 2-4} \cline{6-8}  \multicolumn{1}{c}{}
   &Avg.Iter  &Avg.Time &Opt.val& &Avg.Iter  &Avg.Time  &Opt.val  \\
\midrule
(10,20)    & 9.2   & 0.563  & 6.751   &  & 2.2  & 0.070   & 6.751   \\
(20,20)    & 22.8  & 1.259  & 13.025  &  & 5.0  & 0.098   & 13.025  \\
(22,20)    & 15.8  & 0.951  & 13.974  &  & 4.6  & 0.122   & 13.974  \\
(20,30)    & 17.8  & 1.060  & 11.965  &  & 3.8  & 0.116   & 11.965  \\
(35,50)    & 20.8  & 1.237  & 47.417  &  & 5.4  & 0.202   & 47.417  \\
(45,60)    & 24.0  & 1.364  & 99.224  &  & 7.4  & 0.342   & 99.224  \\
(45,100)   & 24.6  & 1.390  & 89.276  &  & 5.0  & 0.450   & 89.276  \\
(60,100)   & 24.4  & 1.394  & 151.439 &  & 5.4  & 0.548   & 151.439 \\
(70,100)   & 24.8  & 1.463  & 265.115 &  & 7.0  & 0.638   & 265.115 \\
(70,120)   & 32.0  & 1.839  & 218.068 &  & 11.0 & 1.036   & 218.068 \\
(100,100)  & 29.4  & 1.749  & 265.382 &  & 7.0  & 0.870   & 265.382 \\
(100,300)  & 44.0  & 2.987  & 579.726 &  & 13.4 & 7.422   & 579.726 \\
(100,500)  & 54.2  & 3.964  & 558.786 &  & 21.0 & 34.622  & 558.786 \\
(100,800)  & 48.4  & 4.207  & 524.491 &  & 22.2 & 68.512  & 524.491 \\
(100,1000) & 78.2  & 7.409  & 829.049 &  & 37.0 & 144.358 & 829.049 \\
(100,2000) & 67.2  & 10.239 & 908.037 &  & 37.8 & 348.474 & 908.037 \\
(100,3000) & 182.4 & 36.412 & 957.982 &  & $-$  & $-$     & $-$     \\
(100,4000) & 155.6 & 42.948 & 939.095 &  & $-$  & $-$     & $-$     \\
(100,5000) & 206.0 & 71.423 & 952.519 &  & $-$  & $-$     & $-$ \\
\bottomrule
\end{tabular}
\end{table}
\begin{table}[h!]\scriptsize 
\centering\caption{Numerical comparison of OSSBBA with the algorithms in \cite{5B,25B} on P1   \label{TB3}}
\begin{tabular}{lllllllllllllll}
\toprule
\multirow{2}*{\textbf{$(m,n)$}}&\multicolumn{3}{c}{OSSBBA}&
  &\multicolumn{3}{c}{Ref.\cite{5B}}  &\multicolumn{4}{c}{Ref.\cite{25B}}\\
 \cline{ 2-4} \cline{6-8}  \cline{10-12}\multicolumn{1}{c}{}
   &Avg.Iter  &Avg.Time &Opt.val& &Avg.Iter  &Avg.Time  &Opt.val  & &Avg.Iter  &Avg.Time  &Opt.val \\
\midrule
(10,20)    & 14.1  & 0.5425  & 1.5446   &  & 15.0 & 0.2219   & 1.5446   &  & 10.8 & 0.0878   & 1.5446   \\
(20,20)    & 23.0  & 0.6383  & 2.8024   &  & 18.9 & 0.1814   & 2.8024   &  & 14.0 & 0.1032   & 2.8024   \\
(22,20)    & 20.2  & 0.6224  & 4.3818   &  & 13.8 & 0.1304   & 4.3818   &  & 8.6  & 0.0645   & 4.3818   \\
(20,30)    & 25.0  & 0.7144  & 1.8719   &  & 18.0 & 0.1896   & 1.8719   &  & 11.8 & 0.1120   & 1.8719   \\
(35,50)    & 23.3  & 0.7356  & 4.5490   &  & 22.1 & 0.2794   & 4.5490   &  & 16.3 & 0.1982   & 4.5490   \\
(45,60)    & 24.2  & 0.7349  & 9.4198   &  & 19.2 & 0.2819   & 9.4198   &  & 12.4 & 0.1748   & 9.4198   \\
(45,100)   & 22.6  & 0.6932  & 15.7961  &  & 20.2 & 0.3428   & 15.7961  &  & 13.7 & 0.2483   & 15.7961  \\
(60,100)   & 29.8  & 0.9132  & 36.4006  &  & 18.4 & 0.4339   & 36.4006  &  & 12.0 & 0.2752   & 36.4006  \\
(70,100)   & 33.2  & 1.0402  & 57.4905  &  & 21.3 & 0.5418   & 57.4905  &  & 14.6 & 0.4307   & 57.4905  \\
(70,120)   & 31.4  & 1.0304  & 83.7237  &  & 20.8 & 0.5816   & 83.7239  &  & 16.5 & 0.5158   & 83.7238  \\
(100,100)  & 25.2  & 0.8740  & 69.4741  &  & 16.0 & 0.6003   & 69.4741  &  & 11.6 & 0.5435   & 69.4741  \\
(100,300)  & 30.0  & 1.3271  & 88.1626  &  & 21.6 & 2.5660   & 88.1628  &  & 14.4 & 2.0805   & 88.1627  \\
(100,500)  & 51.6  & 2.7543  & 64.9570  &  & 31.5 & 6.0911   & 64.9570  &  & 20.3 & 5.8998   & 64.9570  \\
(100,800)  & 65.2  & 4.1867  & 219.7907 &  & 40.1 & 14.7563  & 219.7907 &  & 23.0 & 14.4794  & 219.7907 \\
(100,1000) & 58.0  & 4.2617  & 63.5351  &  & 37.0 & 18.2963  & 63.5351  &  & 24.8 & 11.6696  & 63.5351  \\
(100,2000) & 98.8  & 12.2297 & 132.1101 &  & 59.6 & 74.1391  & 132.1101 &  & 41.9 & 46.5824  & 132.1101 \\
(100,3000) & 125.5 & 22.4892 & 115.5055 &  & 76.0 & 173.8181 & 115.5055 &  & 52.6 & 103.6231 & 115.5055 \\
(100,4000) & 124.7 & 33.0668 & 127.8224 &  & 68.3 & 260.5574 & 127.8224 &  & 42.2 & 149.0661 & 127.8224 \\
(100,5000) & 164.6 & 55.1585 & 210.0778 &  & 94.4 & 564.6799 & 210.0778 &  & 57.1 & 302.7536 & 210.0778 \\
 \bottomrule
\end{tabular}
\end{table}
\begin{table}[h!]\scriptsize
\caption{Numerical comparison between OSSBBA and BARON on large-scale  P2 \label{TB4} }
\centering
\begin{tabular}{llllllllllll}
\toprule
\multirow{2}*{\textbf{$p$}}&\multirow{2}*{\textbf{$(m,n)$}} &\multicolumn{3}{c}{OSSBBA}& &\multicolumn{3}{c}{BARON}\\
 \cline{ 3-5} \cline{7-9}  \multicolumn{1}{c}{}
   &&Avg.Iter  &Avg.Time &$\ln$(Opt.val)& &Avg.Iter  &Avg.Time  &$\ln$(Opt.val)  \\
 \midrule
2 & (10,1000) & 37.2  & 2.4331   & 3.2272  &  & 40.5  & 52.115   & 3.2272 \\
  & (10,2000) & 45.8  & 3.6634   & 4.3241  &  & 58.4 & 323.756  & 4.3241 \\
3 & (10,1000) & 248.6  & 16.1983   & 6.7526  &  & $-$  & $-$      & $-$       \\
  & (10,2000) & 280.2 & 20.0227  & 7.9307  &  & $-$  & $-$      & $-$       \\
4 & (10,1000) & 1787.2  & 112.3951  & 9.5960  &  & $-$  & $-$      & $-$       \\
  & (10,2000) & 2773.2  & 231.6932  & 12.1181 &  & $-$  & $-$      & $-$       \\
5 & (10,1000) & 6040.9  &356.5587   & 13.9486 &  & $-$  & $-$      & $-$       \\
  & (10,2000) & 7077.7  &528.1786   & 15.4777  &  & $-$  & $-$      & $-$       \\
6 & (10,1000) & 10833.6 &719.1481  & 17.4058 &  & $-$  & $-$      & $-$       \\
  & (10,2000) & 11596.4&1002.3876   &19.3237      &  & $-$  & $-$      & $-$       \\
7 & (10,1000) & 28734.3&2350.9825  &21.9411      &  & $-$  & $-$      & $-$       \\
  & (10,2000) & 36870.1& 3382.1985   &24.3913      &  & $-$  & $-$      & $-$       \\
\bottomrule
\end{tabular}
\end{table}
\begin{table}[h!]\tiny
\centering\caption{Numerical comparison of OSSBBA with the algorithms in \cite{5B,25B} on P2   \label{TB5}}
\begin{tabular}{lllllllllllllll}
\toprule
\multirow{2}*{\textbf{$(p,m,n)$}}&\multicolumn{3}{c}{OSSBBA}&
  &\multicolumn{3}{c}{Ref.\cite{5B}}  &\multicolumn{4}{c}{Ref.\cite{25B}}\\
 \cline{ 2-4} \cline{6-8}  \cline{10-12}\multicolumn{1}{c}{}
   &Avg.Iter  &Avg.Time &$\ln$(Opt.val)& &Avg.Iter  &Avg.Time  &$\ln$(Opt.val)  & &Avg.Iter  &Avg.Time  &$\ln$(Opt.val) \\
\midrule
(2,10,100)   & 28.2  & 0.9808  & 1.0453   &  & 19.0  & 0.1976   & 1.0453   &  & 17.1  & 0.1462   & 1.0453   \\
(2,30,300)   & 44.3  & 1.2296  & 3.1587   &  & 24.6  & 0.5672   & 3.1587   &  & 18.6  & 0.4302   & 3.1587   \\
(2,50,500)   & 45.6  & 1.5945  & 4.7409   &  & 21.3  & 1.4750   & 4.7409   &  & 17.0  & 1.2160   & 4.7409   \\
(2,80,800)   & 55.1  & 3.1248  & 5.4405   &  & 25.4  & 6.2590   & 5.4405   &  & 16.8  & 3.9059   & 5.4405   \\
(2,100,1000) & 45.4  & 4.3807  & 5.9066   &  & 21.2  & 9.7774   & 5.9066   &  & 15.4  & 6.5715   & 5.9066   \\
(2,200,2000) & 51.5  & 28.4199 & 7.2213   &  & 19.4  & 76.5304  & 7.2213   &  & 15.1  & 52.4561  & 7.2214   \\
(3,10,100)   & 155.8 & 5.4686  & 2.6422   &  & 101.9 & 1.1158   & 2.6422   &  & 70.0  & 0.7261   & 2.6422   \\
(3,30,300)   & 173.6 & 5.4350  & 6.3552   &  & 90.0  & 2.6698   & 6.3552   &  & 56.4  & 1.5133   & 6.3552   \\
(3,80,800)   & 264.0 & 16.2792 & 9.3900  &  & 79.8  & 20.0341  & 9.3900  &  & 53.0  & 12.7506  & 9.3900  \\
(3,100,1000) & 214.6 & 18.0938 & 9.8526  &  & 72.0  & 33.6091  & 9.8526  &  & 45.1  & 18.5145  & 9.8526  \\
(3,200,2000) & 256.2 & 109.1915 & 12.5079   &  & 81.4  & 326.9412 &12.5080  &  & 51.4  & 182.3114 & 12.5079   \\
(4,10,100)   & 740.3 & 17.9763 & 6.4970   &  & 283.4 & 2.1674   & 6.4970   &  & 159.6 & 1.1035   & 6.4970   \\
(4,30,300)   & 844.8 & 25.0594 & 9.6558  &  & 292.6 & 7.5720   & 9.6558   &  & 155.7 & 3.8255   & 9.6558  \\
(4,80,800)   & 753.1 & 43.5740 &13.1890 &  & 225.4 & 56.4345  & 13.1890 &  & 115.2 & 26.2765  & 13.1890 \\
(4,100,1000) & 870.4 & 67.3168  &14.3808  &  & 283.9 & 130.4006 & 14.3808 &  & 140.0 & 56.5991  & 14.3808  \\
(4,200,2000) & 921.2 & 432.0423 & 17.0713 &  & 232.8 & 936.3903 & 17.0713  &  & 116.3 & 399.4686 & 17.0713 \\
 \bottomrule
\end{tabular}
\end{table}
\begin{table}[h!]\scriptsize
\caption{ Numerical comparison between OSSBBA and Ref. \cite{25B} on P3 with $p=2$ \label{TB6}}
\centering
\begin{tabular}{llllllllllll}
\toprule
\multirow{2}*{\textbf{$\bar{p}$}}&\multirow{2}*{\textbf{$(m,n)$}}
 &
 &\multicolumn{1}{c}{OSSBBA}
 &\multicolumn{5}{c}{Ref. \cite{25B}}\\
\cline{3-5} \cline{7-9} \multicolumn{1}{c}{}
 &&Avg.Iter&Avg.Time&Opt.val &&Avg.Iter&Avg.Time&Opt.val\\
 \midrule
1 & (10,100)   & 10.0 & 5.4211   & 0.0110   &  & 11.6 & 0.1138   & 0.0110   \\
1 & (20,200)   & 11.6 & 8.3662   & 0.0692   &  & 43.0 & 0.9960   & 0.0692   \\
1 & (30,300)   & 13.0 & 7.5526   & 0.0800   &  & 70.2 & 4.1061   & 0.0800   \\
1 & (40,400)   & 14.0 & 8.9950   & 0.1430   &  & 69.6 & 8.4874   & 0.1430   \\
1 & (50,500)   & 11.6 & 9.1429   & 0.0987   &  & 53.6 & 13.4372  & 0.0987   \\
1 & (80,800)   & 9.8  & 11.4018  & 0.3060   &  & 41.6 & 41.0454  & 0.3060   \\
1 & (100,1000) & 10.0 & 13.1412  & 0.0114   &  & 37.6 & 49.0706  & 0.0114   \\
1 & (200,2000) & 17.6 & 117.8562 & 0.0013   &  & 55.4 & 302.1039 & 0.0013   \\
1 & (300,3000) & 10.0 & 158.2001 & 0.8862   &  & 31.6 & 925.3419 & 0.8862   \\
2 & (10,100)   & 16.1 & 1.3095   & 0.3988   &  & 5.2  & 0.0603   & 0.3988   \\
2 & (20,200)   & 61.2 & 2.2489   & 1.2245   &  & 14.6 & 0.4092   & 1.2245   \\
2 & (30,300)   & 56.0 & 1.7469   & 2.0752   &  & 15.0 & 0.8826   & 2.0752   \\
2 & (40,400)   & 69.4 & 2.1826   & 8.3702   &  & 21.8 & 2.6910   & 8.3702   \\
2 & (50,500)   & 46.8 & 3.1453   & 0.3755   &  & 13.1 & 3.2455   & 0.3755   \\
2 & (80,800)   & 53.0 & 4.8080   & 0.2231   &  & 16.9 & 17.8305  & 0.2231   \\
2 & (100,1000) & 57.6 & 6.4876   & 7.0978   &  & 15.4 & 24.8842  & 7.0978   \\
2 & (200,2000) & 66.2 & 31.4479  & 165.4504 &  & 16.6 & 142.7226 & 165.4512 \\
2 & (300,3000) & 74.4 & 135.8912 & 5.9284   &  & 16.6 & 444.1219 & 5.9284  \\

\bottomrule
\end{tabular}
\end{table}
\begin{table}[h!]\scriptsize
\caption{ Numerical comparison between OSSBBA and Ref. \cite{25B} on P3 with $p=3$ \label{TB7}}
\centering
\begin{tabular}{llllllllllll}
\toprule
\multirow{2}*{\textbf{$\bar{p}$}}&\multirow{2}*{\textbf{$(m,n)$}}
 &
 &\multicolumn{1}{c}{OSSBBA}
 &\multicolumn{5}{c}{Ref. \cite{25B}}\\
\cline{3-5} \cline{7-9} \multicolumn{1}{c}{}
 &&Avg.Iter&Avg.Time&Opt.val &&Avg.Iter&Avg.Time&Opt.val\\
 \midrule
1 & (10,100)   & 14.6  & 6.2924   & 0.0073     &  & 84.4  & 0.8711    & 0.0073     \\
1 & (20,200)   & 14.9  & 7.8065   & 0.0084     &  & 104.7 & 2.5243    & 0.0084     \\
1 & (30,300)   & 13.0  & 7.4576   & 0.2038     &  & 155.1 & 6.9968    & 0.2038     \\
1 & (40,400)   & 14.4  & 10.2515  & 0.0035     &  & 259.4 & 30.6222   & 0.0035     \\
1 & (50,500)   & 14.8  & 11.3020  & 0.0007     &  & 304.2 & 74.0663   & 0.0007     \\
1 & (80,800)   & 14.6  & 15.9440  & 0.0136     &  & 182.6 & 143.9457  & 0.0136     \\
1 & (100,1000) & 14.8  & 21.2610  & 0.0008     &  & 151.0 & 195.1263  & 0.0008     \\
1 & (200,2000) & 14.0  & 85.2749  & 0.2153     &  & 179.2 & 1096.0737 & 0.2153     \\
1 & (300,3000) & 14.5  & 272.1227 & 0.0145     &  & $-$   & $-$       & $-$        \\
2 & (10,100)   & 78.8  & 39.6242  & 0.0987     &  & 136.7 & 1.2844    & 0.0987     \\
2 & (30,300)   & 112.0 & 63.8271  & 0.0086     &  & 226.6 & 13.4809   & 0.0086     \\
2 & (50,500)   & 101.8 & 64.0586  & 0.1519     &  & 208.8 & 54.7808   & 0.1519     \\
2 & (80,800)   & 111.6 & 98.0348  & 0.7200     &  & 245.4 & 164.6378  & 0.7200     \\
2 & (100,1000) & 130.1 & 107.5559 & 0.4779     &  & 161.8 & 346.0197  & 0.4779     \\
2 & (200,2000) & 101.8 & 321.0563 & 3.8493     &  & 183.1 & 2056.3639 & 3.8493     \\
2 & (300,3000) & 96.0  & 689.8087 & 84.2601    &  & $-$   & $-$       & $-$        \\
3 & (10,100)   & 234.6 & 6.5275   & 0.7453     &  & 39.4  & 0.4109    & 0.7453     \\
3 & (30,300)   & 261.8 & 7.4203   & 0.9984     &  & 50.0  & 3.3524    & 0.9984     \\
3 & (50,500)   & 308.0 & 10.6432  & 39.3818    &  & 74.8  & 19.6266   & 39.3818    \\
3 & (80,800)   & 306.6 & 16.0365  & 94.7312    &  & 75.2  & 78.5351   & 94.7312    \\
3 & (100,1000) & 527.6 & 36.0380  & 1349.0448  &  & 115.6 & 207.1581  & 1349.0581  \\
3 & (200,2000) & 648.8 & 188.7801 & 17970.9364 &  & 143.2 & 1564.3865 & 17971.2922\\

\bottomrule
\end{tabular}
\end{table}
\begin{table}[p!]\scriptsize
\caption{ Numerical comparison between OSSBBA and Ref. \cite{25B} on P3 with $p=4,5$ \label{TB8}}
\centering
\begin{tabular}{llllllllllll}
\toprule
\multirow{2}*{\textbf{$\bar{p}$}}&\multirow{2}*{\textbf{$(p,m,n)$}}
 &
 &\multicolumn{1}{c}{OSSBBA}
 &\multicolumn{5}{c}{Ref. \cite{25B}}\\
\cline{3-5} \cline{7-9} \multicolumn{1}{c}{}
 &&Avg.Iter&Avg.Time&Opt.val &&Avg.Iter&Avg.Time&Opt.val\\
 \midrule
1 & (4,10,100)   & 13.8    & 10.4223 & 0.0000     &  & 139.0  & 1.1912    & 0.0000     \\
1 & (4,30,300)   & 15.8    & 9.3492           & 0.0353     &  & 381.1  & 21.8723   & 0.0353     \\
1 & (4,50,500)   & 15.2    & 11.3419          & 0.0001     &  & 439.2  & 89.0626   & 0.0001     \\
1 & (4,80,800)   & 17.6    & 18.0001          & 0.0006     &  & 416.4  & 229.2804  & 0.0006     \\
1 & (4,100,1000) & 17.2    & 23.2597          & 0.0001     &  & 394.8  & 251.8559  & 0.0001     \\
2 & (4,10,100)   & 50.5    & 24.0921          & 0.0900     &  & 413.4  & 3.5112    & 0.0900     \\
2 & (4,30,300)   & 108.8   & 55.4369          & 0.1031     &  & 354.2  & 23.6881   & 0.1031     \\
2 & (4,50,500)   & 139.8   & 92.4701          & 0.0530     &  & 388.8  & 104.6994  & 0.0530     \\
2 & (4,80,800)   & 122.6   & 123.1584         & 0.0098     &  & 397.7  & 252.2755  & 0.0098     \\
2 & (4,100,1000) & 109.5   & 127.7022         & 0.0000     &  & 565.8  & 872.2801  & 0.0000     \\
4 & (4,10,100)   & 951.6   & 25.0941          & 3.2922     &  & 83.2   & 0.7800    & 3.2922     \\
4 & (4,30,300)   & 2533.2  & 79.6079          & 1.0912     &  & 242.8  & 17.7726   & 1.0912     \\
4 & (4,50,500)   & 2663.4  & 91.9915          & 3.5358     &  & 283.2  & 81.6886   & 3.5358     \\
4 & (4,80,800)   & 2385.0  & 126.9785         & 9579.8666  &  & 248.4  & 239.9699  & 9579.8668  \\
4 & (4,100,1000) & 3131.8  & 217.3545         & 4.4266     &  & 327.8  & 648.6133  & 4.4266     \\
1 & (5,10,100)   & 15.4    & 8.1372           & 0.0000     &  & 157.0  & 1.1772    & 0.0000     \\
1 & (5,30,300)   & 14.8    & 8.3628           & 0.0000     &  & 742.4  & 48.5301   & 0.0000     \\
1 & (5,50,500)   & 17.2    & 12.3179          & 0.0000     &  & 795.6  & 114.1040  & 0.0000     \\
1 & (5,80,800)   & 18.4    & 19.4651          & 0.0001     &  & 307.2  & 105.7859  & 0.0001     \\
1 & (5,100,1000) & 14.6    & 23.1413          & 0.0000     &  & 156.4  & 99.0862   & 0.0000     \\
3 & (5,10,100)   & 856.6   & 416.5373         & 1.0672     &  & 2353.0 & 22.0411   & 1.0672     \\
3 & (5,30,300)   & 1827.6  & 975.8683         & 0.0190     &  & 6588.2 & 479.5366  & 0.0190     \\
3 & (5,50,500)   & 1053.6  & 673.0753         & 0.3176     &  & 4090.4 & 862.9126  & 0.3176     \\
3 & (5,80,800)   & 920.8   & 830.4220         & 0.1186     &  & 2890.2 & 2084.5160 & 0.1186     \\
3 & (5,100,1000) & 1076.3  & 1022.3382        & 0.5510     &  & 0.0    & 0.0000    & 0.0000     \\
5 & (5,10,100)   & 7174.0  & 184.4887         & 9.5227     &  & 240.6  & 2.1863    & 9.5227     \\
5 & (5,30,300)   & 12280.6 & 392.9769         & 11.1198    &  & 456.4  & 32.8679   & 11.1198    \\
5 & (5,50,500)   & 10618.4 & 379.4513         & 947.9276   &  & 573.8  & 158.0420  & 947.9276   \\
5 & (5,80,800)   & 15302.3 & 791.5586         & 180.9346   &  & 587.7  & 612.5000  & 180.9346   \\
5 & (5,100,1000) & 16761.8 & 967.8614         & 11282.6541 &  & 731.2  & 1392.7121 & 11282.6542\\

\bottomrule
\end{tabular}
\end{table}

Overall, BARON is only more suitable for solving small-scale GLMP problems than our algorithm, as shown in some numerical results in Tables \ref{TB1}, \ref{TB2}, \ref{TB4}.
More bluntly, the results in these three Tables show that the computational power of Algorithm  OSSBBA  is weaker than that of BARON in solving most small-scale problems (especially when $p= 2$). However,
However, the numerical results in Table \ref{TB2} reveal that at least when the number of variables $n$ exceeds 300, our algorithm is more computable than BARON.
Naturally, the numerical results in Table \ref{TB4} additionally demonstrate this for different sizes of $p$. Especially when $p\geq3$, neither BARON can solve all problems in 3600s.  Although it is not appropriate for dealing with large-scale P2 problems, this package undoubtedly has the potential to address some P2 problems on a smaller scale.
Nevertheless, it becomes clear that they are extremely close when comparing the optimal values of Algorithm OSSBBA with the software. This serves for verifying the feasibility of our algorithm once more.

Table \ref{TB3} compares the numerical results of the OSSBBA algorithm with the two in \cite{5B,25B} when solving the P1 problem.
For the small-scale problems in the first 11 groups, it can be viewed that our algorithm has marginally less computational power than the other two algorithms, both in terms of the number of iterations and the CPU runtime during execution.
While our algorithm still requires more iterations than the other two when $n\geq3$, it uses less CPU time overall.
This is primarily due to the fact that for each iteration of the OSSBBA algorithm, only one linear programming problem and two small-scale convex  programming problems must be solved, while for the other two algorithms, two larger linear programs have to be addressed.
The same phenomenon is reflected in the numerical results of problem P2 in Table \ref{TB5}, nevertheless for the reason the size of $p$ has a significant impact on the performance of the three algorithms, the scope of application of the OSSBBA is limited, but it is still more suitable than the two algorithms in \cite{5B,25B} especially for particular instances of the large-scale problem, such as when $(m,n)=(100,1000),(200,2000)$, the amount of time required via the algorithm in \cite{5B} is nearly twice or three times that of our algorithm.

Essentially, when the objective  function of the problem GLMP possesses both positive and negative exponents, the algorithm only branches in the $\bar{p}$-dimensional space, where  $\bar{p}$ is defined before Eq. (\ref{Jeq1}) and cannot be greater than $p$.
From Tables \ref{TB3} and \ref{TB5}, it is known that the algorithm in \cite{25B} is always more efficient than that in  \cite{5B}, so that we compare only the OSSBBA with the former in the remainder of the numerical results   on the problem  P3 (see  Tables \ref{TB6}-\ref{TB8}).

Looking at Tables  \ref{TB6}, \ref{TB7} and \ref{TB8}, our algorithm performs better when $\bar{p}$ is far from $p$, while the algorithm in \cite{25B} is the opposite.  However,  it must be pointed out that in the case of $\bar{p}=p$, the numerical results in Tables  \ref{TB6}-\ref{TB8} show similar phenomena as in Tables  \ref{TB1}-\ref{TB5}. That is, although the algorithm in \cite{25B} has fewer iterations than our algorithm, it is only more suitable for solving small-scale problems.
At the same time, the linear relaxation subproblems of the algorithm in \cite{25B} is derived from the convex envelope of natural logarithms, which is the tightest relaxation; besides, our algorithm does not solve the convex optimization problem, but only solves a linear programming problem (\ref{zeqj}) of the same size as GLMP and two small-scale convex optimization problems (\ref{eqj5}). As a result, when solving certain large-scale problems P3 with $\bar{p}=p$ (e.g., $(p,m,n)=(2,100,1000)$, $(2,200,2000)$, $(2,300,3000)$, $(3,100,1000)$, $(3,200,2000)$, $(4,100,1000)$, $(5,100,1000)$), although our  Avg.Iter is higher compared to the algorithm in \cite{25B}, the Avg.Time is lower than the latter.

Further, when $\bar{p}<p$, the subproblem (\ref{zeqj}) solved by our algorithm becomes a convex optimization problem, but the branching operation acts in a $\bar{p}$-dimensional space, resulting in very few iterations of our algorithm. On the other hand, another algorithm branches in a $p$-dimensional space, which requires a large number of iterations.
At this time, due to the inability of the convex envelope to be adopted on the natural logarithms corresponding to negative exponents, the advantage of the relaxation subproblem in the algorithm \cite{25B} no longer exists.
 In terms of CPU running time, algorithm OSSBBA seems more suitable for solving some large-scale problems in all cases of problem P3 compared to another algorithm.

 In addition to the above analysis, it has to be mentioned that the numerical results in the above tables illustrate that the larger the number of linear terms $p$ in the GLMP problem, the worse the computational efficiency of the two algorithms of Jiao et al.
 Correspondingly, the computational efficiency of the algorithm OSSBBA is only related to the number of positive exponents $\bar{p}$, and is not directly related to the size of $p$ (e.g., in the case of $(\bar{p},p)=(1,5)$, our algorithm is much better than the algorithm in \cite{25B}).
 Besides,  when $\bar{p}$ and $p$ is fixed, as $(m,n)$ increase, both CPU time and iterations  increase.
 Note that, the algorithm in \cite{25B} is more suitable for solving small-scale GLMP problems with positive exponents than our algorithm.
 These facts are clarifying that different algorithms have their own solving capabilities and scope of application.

In summary, our algorithm can effectively solve GLMP-like problems, especially for larger-scale problems, it has stronger computational efficiency than the algorithms in \cite{5B,25B} and  the commercial software package BARON.

\section{Conclusions}
This paper introduces a global algorithm for solving the GLMP problem by incorporating the new convex relaxation subproblem, simplex branching operation and pruning operation into the BB framework.
Under the given tolerance, the finite global convergence of the algorithm and its computational complexity analysis are proved in detail.
 Numerical experiments demonstrate that the proposed algorithm is feasible for all test problems, and saves more computational time than the algorithms in BARON and \cite{5B,25B} when solving specific large-scale GLMP instances.
  Furthermore, it is observed that the number of iterations and CPU time gradually increase as the number of positive exponents in the objective function increases.
Future research will focus on developing effective global algorithms for more general GLMP problems using new relaxation strategies and BB techniques.


\section*{Funding}
This research is supported by the National Natural Science Foundation of China under Grant [grant number 12301401].

\section*{Data availability statement}
Supporting data for this study are available from the corresponding author, as reasonably requested.

 \begin{acknowledgements}
We thank each anonymous reviewer for their valuable comments and suggestions, which will help improve the quality of the papers.
\end{acknowledgements}


\end{document}